\documentclass[12pt,leqno,fleqn]{amsart}
\usepackage{amsmath,amstext,amsthm,amssymb,amsxtra}
\usepackage[normalem]{ulem}

\usepackage{txfonts} 
\usepackage[T1]{fontenc}
\usepackage{lmodern}

\usepackage{euler}  

\usepackage{graphicx}
\usepackage{transparent}
\usepackage{todonotes}
\usepackage{mathtools}
\mathtoolsset{showonlyrefs,showmanualtags}

\usepackage
[
bookmarksopen,
bookmarksdepth=3,
colorlinks,
citecolor=red,
pagebackref,
hypertexnames=true,
]
{hyperref}

\usepackage
[
backrefs,
msc-links,
nobysame,
initials,
non-sorted-cites
]
{amsrefs}

\usepackage{url}

\DefineSimpleKey{bib}{myurl}
\newcommand\myurl[1]{\url{#1}}

\BibSpec{webpage}{%
  +{}{\PrintAuthors} {author}
  +{,}{ \textit} {title}
  +{}{ \parenthesize} {date}
  +{,}{ \myurl} {myurl}
  +{,}{ } {note}
  +{.}{ } {transition}
}

\allowdisplaybreaks

\theoremstyle{plain}

\newtheorem{lemma}[equation]{Lemma}
\newtheorem{proposition}[equation]{Proposition}

\newtheorem{theorem}[equation]{Theorem}
\newtheorem*{theorem*}{Theorem}
\newtheorem{corollary}[equation]{Corollary}
\theoremstyle{definition}
\newtheorem{definition}[equation]{Definition}

\newtheorem{remark}[equation]{Remark}

\numberwithin{equation}{section}

\def\norm#1.#2.{\lVert#1\rVert_{#2}}
\def\Norm#1.#2.{\bigl\lVert#1\bigr\rVert_{#2}}
\def\NOrm#1.#2.{\Bigl\lVert#1\Bigr\rVert_{#2}}
\def\NORm#1.#2.{\biggl\lVert#1\biggr\rVert_{#2}}
\def\NORM#1.#2.{\Biggl\lVert#1\Biggr\rVert_{#2}}

\def\ind{\textnormal{\textbf 1}}

\def\R{\mathbb R}

\DeclareMathOperator*{\esssup}{ess\,sup}
\DeclareMathOperator*{\essinf}{ess\,inf}

\def\MQ{\mathsf{M}_Q }
\def\M{\mathsf{M}}

\def\ip#1,#2,{\langle #1,#2\rangle}
\def\Ip#1,#2,{\bigl\langle#1,#2\bigr\rangle}
\def\IP#1,#2,{\Bigl\langle#1,#2\Bigr\rangle}

\begin{document}
\raggedbottom
 \title[Sharp reverse H\"older inequalities for flat weights]{Asymptotically sharp reverse H\"older inequalities for flat Muckenhoupt weights}

\author{Ioannis Parissis}
\address{I.P.: Departamento de Matem\'aticas, Universidad del Pais Vasco, Aptdo. 644, 48080
Bilbao, Spain and Ikerbasque, Basque Foundation for Science, Bilbao, Spain}
\email{\href{mailto:ioannis.parissis@ehu.es}{ioannis.parissis@ehu.es}}
\thanks{I. P. is supported by grant  MTM2014-53850 of the Ministerio de Econom\'ia y Competitividad (Spain), grant IT-641-13 of the Basque Government, and IKERBASQUE}
\thanks{E.R. is partially supported by grants UBACyT 20020130100403BA, PIP (CONICET) 11220110101018, by the Basque Government through the BERC 2014-2017 program, and by the Spanish Ministry of Economy and Competitiveness MINECO: BCAM Severo Ochoa accreditation SEV-2013-0323.}

\author{Ezequiel Rela}
\address{E.R.: Departamento de Matem\'atica,
Facultad de Ciencias Exactas y Naturales,
Universidad de Buenos Aires, Ciudad Universitaria
Pabell\'on I, Buenos Aires 1428 Capital Federal Argentina} \email{\href{mailto:erela@dm.uba.ar}{erela@dm.uba.ar}}

\begin{abstract}  We present reverse H\"older inequalities for Muckenhoupt weights in $\mathbb{R}^n$ with an asymptotically sharp behavior for flat weights, namely $A_\infty$ weights with Fujii-Wilson constant $(w)_{A_\infty}\to 1^+$. That is, the local integrability exponent in the reverse H\"older inequality blows up as the weight becomes nearly constant. This is expressed in a precise and explicit computation of the constants involved in the reverse H\"older inequality. The proofs avoid BMO methods and rely instead on precise covering arguments. Furthermore, in the one-dimensional case we prove sharp reverse H\"older inequalities for one-sided and two sided weights in the sense that both the integrability exponent as well as the multiplicative constant appearing in the estimate are best possible. We also prove sharp endpoint weak-type reverse H\"older inequalities and consider further extensions to general non-doubling measures and multiparameter weights.
\end{abstract}

\maketitle

\section{Introduction and main results}  
\subsection {Main definitions and background.} A central topic in modern harmonic analysis is the problem of finding precise norm estimates for general operators defined on weighted Lebesgue spaces $L^p(w)$, where $w$ is \emph{a weight}: a  locally integrable function $w:\R^n\to \R$ with $w\geq 0$. The natural class that arises when studying typical operators such as the Hardy-Littlewood maximal operator, the Hilbert transform or, more generally, Calder\'on-Zygmund operators, is the Muckenhoupt class $A_p$. It is defined as the  class of weights $w$ such that
\begin{equation}\label{eq.Ap}
[w]_{A_p}\coloneqq \sup_Q\bigg(\fint_Q w \bigg) \bigg(\fint_Q w^{-\frac{1}{p-1}}\bigg)^{p-1}<\infty, \qquad 1<p<\infty.
\end{equation}
Here the supremum is computed over all cubes $Q$ in $\mathbb{R}^n$ with sides parallel to the coordinate axes. We will often write $w_Q$ for $\fint_Q w:= w(Q)/|Q|$ where $|E|$ stands for the Lebesgue measure of the set $E\subset \mathbb{R}^n$ and $w(E)=\int_E w$. 
The limiting case of \eqref{eq.Ap} when $p\to 1^+$ defines the class $A_1$ which consists of weights $w$ such that
\[
[w]_{ A_1 }:=\sup_{Q}\bigg(\fint_Q w \bigg) \esssup_{Q} (w^{-1})<+\infty.
\]
This is equivalent to $[w]_{A_1}$ being the smallest positive constant $C$ such that
\[
 \M w(x)\le C w(x)\qquad \text{ a.e. } x \in \mathbb{R}^n.
\]
Here $\M$ denotes the non-centered Hardy-Littlewood maximal operator on $\R^n$
\[
\M f (x)\coloneqq \sup_{Q\ni x }\fint_Q |f(y)|\ dy,\qquad x\in \R^n ,
\]
where the supremum is taken with respect to all cubes $Q$ in $\R^n$ with sides parallel to the coordinate axes. The importance of these weight classes arises from the fact that they characterize the boundedness of many of the aforementioned operators, as for example, the Hardy-Littlewood maximal operator, the Hilbert transform, and the Riesz transforms, on $L^p(w)$.

The union of all $A_p$ classes is denoted by $A_\infty$. While this is a qualitative description of $A_\infty$ many quantitative characterizations are also known. For example, a weight $w$ is in the Muckenhoupt class $A_\infty$ if there exist constants $c,\eta>0$ such that for every cube $Q$ and every measurable $E\subseteq Q$ we have 
\begin{equation}\label{eq.Ainfty}
\frac{w(E)}{w(Q)}\leq c \big(\frac{|E|}{|Q|}\big)^\eta.
\end{equation}
In accordance with \eqref{eq.Ap}, there are a couple of different constants that can be used to gauge  $A_\infty$. The most prominent choices in the literature consist of the Khrushchev constant and the Fujii-Wilson constant. The former was introduced by Khrushchev in \cite{Hru} and arises by taking the formal limit in \eqref{eq.Ap} as $p\to+\infty$
\begin{equation}\label{eq.AinftyK}
[w]_{A_\infty} \coloneqq \sup_Q \bigg(\fint_Q w \bigg) \bigg(\exp \fint_Q \log w^{-1} \bigg).
\end{equation}
Another constant that can be used to quantify the $A_\infty$-property was implicitly introduced by Fujii in \cite{F}, and later rediscovered by Wilson, \cites{W,W1}, and it is defined as
\begin{equation}\label{eq.AinftyW}
(w)_{A_\infty}\coloneqq \sup_Q \frac{1}{w(Q)}\int_Q \M(w\ind_Q).
\end{equation}
The $A_\infty$ constant defined above is usually referred to as the Fujii-Wilson $A_\infty$ constant. We then have that $w\in A_\infty \Leftrightarrow (w)_{A_\infty}<+\infty\Leftrightarrow [w]_{A_\infty}<+\infty$.

Any of the above three conditions \eqref{eq.Ainfty}, \eqref{eq.AinftyK}, and \eqref{eq.AinftyW}, can be used to quantify the self improving property of Muckenhoupt weights, typically expressed via a reverse H\"older inequality. That is, any $A_\infty$ weight $w$ satisfies an inequality of the form
\begin{equation}\label{eq.general-RHI-intro}
\fint_{Q} w^{1+\varepsilon}\ dx \le C\left(\fint_{Q} w \ dx\right)^{1+\varepsilon},
\end{equation}
valid for every cube $Q$, and the constants $C,\varepsilon>0$ appearing in the inequality above can be related to the $A_\infty$ constant of the weight.  The aim of a quantitative analysis involving $A_\infty$-type conditions is to look for precise versions of inequality \eqref{eq.general-RHI-intro} where the dependence of the constants $C,\epsilon>0$ on the $A_\infty$ constant of choice is explicit and, if possible, sharp. This type of result is important for applications as, typically, one part of the $L^p(w)$-norm of standard operators in harmonic analysis depends on the self improving property of the weight, quantified by $A_\infty$. Indeed, a large variety of improvements on weighted inequalities for maximal and singular integral operators were achieved recently by relying on sharp reverse H\"older inequalities for $A_\infty$ weights. See for example \cites{HytLa,HP,Lerner}.

Sharp reverse H\"older inequalities exist in the literature mainly in dimension $n=1$. The prototypes of sharp reverse H\"older inequalities are arguably the $A_1$-result of Bojarski, Sbordone, and Wik, \cite{BSW}, and the $A_p$-results of Vasyunin for $1\leq p\leq \infty$, \cite{Vas}. In these cases, both the range of integrability as well as the multiplicative constant is calculated exactly as a function of the $A_p$-constant. We note here that, in contrast with the present paper, the $A_\infty$-results from \cite{Vas} are given in terms of the Khrushchev constant $[w]_{A_\infty}$.

For $A_p$ and $A_\infty$ weights in higher dimensions the only sharp result known is the reverse  H\"older inequality for dyadic $A_1$ weights due to Melas, \cite{Melas}, and the corresponding dyadic endpoint version of Os{\c{e}}kowski, \cite{Ose}. On the other hand, the main theorem of Kinnunen in \cite{Kin} gives an asymptotically sharp reverse H\"older inequality for $A_1$ weights in any dimension. In a similar spirit, an asymptotically sharp reverse H\"older inequality for $A_\infty$ in any dimension is given by Korey in \cite{K}, in terms of the Khrushchev constant $[w]_{A_\infty}$. Here asymptotically sharp means that the range of local integrability increases to $\infty$ as $[w]_{A_1}\to 1^+$ or $[w]_{A_\infty}\to 1^+$. In another direction, the analog of the sharp reverse H\"older inequality of \cite{BSW} is extended to \emph{multiparameter} $A_1 ^*$ weights in \cite{Kin}. The multiparameter weight-classes $A_p^*$ are defined analogously to \eqref{eq.Ap}, with the averages over cubes being replaced by averages over rectangular parallelepipeds in $\R^n$ with sides parallel to the coordinate axes. We mention however that multiparameter weights in $A_1 ^*$  constitute a much smaller class than one-parameter $A_1$ in any dimension $n>1$.

Concerning reverse H\"older inequalities for $A_\infty$ in terms of the Fujii-Wilson constant, the most relevant result is from Hyt\"onen and Perez, \cite{HP}, where the authors prove \eqref{eq.general-RHI-intro} for any $0\le \varepsilon\leq  (c_n(w)_{A_\infty} -1)^{-1}$, for some large dimensional constant $c_n>1$. Although this type of result is sufficient for applications to mixed weighted bounds for maximal operators, it does not recover the correct range of integrability for weights $w$ for which $(w)_{A_\infty}\to 1^+$.

Finally, we mention here that the precise asymptotic behavior of weighted operator norms $\|T\|_{L^p(w)}$ as $[w]_{A_\infty}\to 1^+$, for a large class of operators $T$ in harmonic analysis, is contained in the results of Pattakos and Volberg, \cites{pattakos-MRL,pattakos-PAMS}. In particular these results show that one typically recovers the unweighted bounds as the weight becomes almost constant, in the sense described above.

\subsection{Description of main results}The purpose of this paper is to provide a precise quantitative analysis of reverse H\"older inequalities for $A_\infty$ weights, 
where the local integrability exponent and the multiplicative constant are given as explicit functions of the Fujii-Wilson constant $(w)_{A_\infty}$. We are particularly interested in the asymptotic behavior of the local integrability exponent in the reverse H\"older inequality for weights $w$ for which $(w)_{A_\infty}\to 1^+$ and we colloquially refer to such weights as \emph{flat weights}.

The heuristic behind the terminology \emph{flat weight} is that $w$ is ``nearly constant'' whenever $(w)_{A_\infty}\le 1+\eta$ for small $\eta>0$. Since the limit case $\eta=0$ corresponds exactly to a constant weight $w\equiv c$, it is reasonable to expect the local integrability exponent in the reverse  H\"older inequality for $A_\infty$ weights  will grow to infinity as $(w)_{A_\infty}\to 1^+$. 

A particularly convenient albeit not so rigorous aspect of the definition of the constant $(w)_{A_\infty}$ is that, given any cube $Q$ in $\R^n$, one can estimate
\[
\frac{\M(w\ind_Q)(Q)}{|Q|}\leq (w)_{A_\infty}\frac{w(Q)}{|Q|}\leq (w)_{A_\infty} \inf_Q \M(w\ind_Q).
\]
We interpret the estimate above as a \emph{local $A_1$-condition} on $\M(w\ind_Q)$. It is exactly the localization that is required in the definition of $(w)_{A_\infty}$ that makes the above point of view non-rigorous. However, the heuristic given by this simple observation allows us, in most cases, to conclude reverse H\"older inequalities for $\M(w\ind_Q)$, and thus for $w$ on $Q$, as if $\M(w\ind_Q)$ were an $A_1$ weight with $A_1$-constant equal to $(w)_{A_\infty}$. Our results also show that $A_\infty$ already encodes fully the self improving properties of weights and that one doesn't gain any further integrability by assuming, for example, $A_1$ information.

Our first main result provides an asymptotically sharp reverse H\"older inequality in terms of the Fujii-Wilson's constant \emph{in all dimensions}. In what follows, $r'$ is the dual exponent of $r$, that is $1/r+1/r'=1$.
\begin{theorem}\label{t.inftyn} Let $w\in A_\infty$ with  $(w)_{A_\infty}\leq\delta$ for some $\delta\geq 1$ and let $Q$ be a cube in $\mathbb R^n$. Then for all $1\leq r <1+\frac{1}{2^n(\delta-1)}$ we have the reverse H\"older inequality
\[
\fint_Q w^r \leq \delta \frac{r'-1}{r'-1-2^n(\delta-1)}\big(\fint_Q w\big)^r.
\]
\end{theorem}

The theorem above shows that for $w\in A_\infty$, there exist  $r_w=r_w((w)_{A_\infty})>1 $ and  a function $[1,r_w)\ni r\mapsto C_w(r)\in[1,\infty)$ such that for all $1\leq r<r_w$ we have
	\[
 \fint_Q w^{r}  \leq C_w(r)  \big(\fint_Q w\big)^r.
	\]
The value of this estimate is that it is \emph{asymptotically sharp} as we have
\[
\lim_{(w)_{A_\infty}\to 1^+}r_w((w)_{A_\infty})=+\infty,\quad{\text{and}}\quad\lim_{(w)_{A_\infty}\to 1^+} C_w(r)=1,
\]
for any fixed $r$ in the allowed range.
 
In one dimension we prove a sharp reverse H\"older inequality for $A_\infty$ weights in terms of the Fujii-Wilson constant. It is important to note the similarity between the result below and the sharp reverse H\"older inequality for $A_1$ weights in one dimension from \cite{BSW}. This supports the understanding of $\M(w\ind_Q)$, with $w\in A_\infty$, as a local $A_1$ weight.
\begin{theorem}\label{t.RHI-1d} Let $w\in A_\infty$ with Fujii-Wilson constant $(w)_{A_\infty}\leq \delta$ for some $\delta\geq 1$. For all  $1\leq r <	\frac{\delta}{\delta-1}$ we have the reverse H\"older inequality
	\[
	\fint_I \M(w\ind_I)^r \leq  \frac{1}{\delta^{r-1}}\frac{r'-1}{r'-\delta} \big(\fint_I \M(w\ind_I)\big)^r
	\]
for all finite intervals $I$ on the real line. Both the range of integrability as well as the multiplicative constant are best possible. 
\end{theorem}
We also prove the corresponding sharp reverse H\"older inequalities for \emph{one-sided} Muckenhoupt weights in $A_1 ^+$ and $A_\infty ^+$. For precise definitions and statements see \S\ref{s.onesided}. 

Finally, in one dimension, we ask whether one can have a weak-type reverse H\"older inequality for $A_1$ or $A_\infty$ at the endpoint, where the strong-type $L^r$-norm on the left hand side is replaced by the weak-type $L^{r,\infty}$-norm. Indeed such endpoint estimates are true and the following theorem contains an example of such an estimate.
\begin{theorem}\label{t.A1endpoint} Let $w\in A_1 $ with $[w]_{A_1}\leq \delta$ for some $\delta>1$, and let $r_w\coloneqq \delta/(\delta-1)$. Then for every bounded interval $I$ we have that $w\in L^{r_w,\infty}(I,\frac{dx}{|I|})$ and
	\[
	\|w\|_{L^{r_w,\infty}(I,\frac{dx}{|I|})} \leq   \fint_I w .
	\]
Furthermore, the range of weak-type integrability and the multiplicative constant $1$ are best possible.
\end{theorem}

Before proceeding to the proofs of our results, some further remarks are in order.
\begin{remark} We remember here that for $\varepsilon\in(0,2)$, the Gurov-Reshetnyak class $GR_\varepsilon$ is defined as the class of weights $w$ such that
\[
\sup_Q \frac{1}{w(Q)} \int_Q|w-w_Q|\leq \varepsilon.
\]
 If $w\in A_\infty$ we have the easy estimate
\begin{equation}\label{eq.grainfty}
	\fint_Q |w-w_Q|\leq \fint_Q |\MQ(w\ind_Q)-w|+\fint_Q |\MQ(w\ind_Q)-w_Q|\leq 2((w)_{A_\infty}-1)w_Q
\end{equation}
and thus $A_\infty \subseteq \mathsf{GR}_{2((w)_{A_\infty}-1)}$. Then one can use the results in \cite{KLS} in order to conclude that $w$ satisfies an asymptotically sharp reverse H\"older inequality with exponent $r$, where $r-1\eqsim  ((w)_{A_\infty}-1)^{-1}$ as $(w)_{A_\infty}\to 1^+$. Observe that this is the same asymptotic behavior as the one obtained in Theorem~\ref{t.inftyn} for $(w)_{A_\infty}\to 1^+$. Our proof however is direct and conceptually simpler.

In the dyadic case, observation \eqref{eq.grainfty} also yields an asymptotically sharp reverse H\"older inequality by exploiting the estimate in \cite{BKM}*{Corollary 3.10}. The proof of \cite{BKM}*{Corollary 3.10} gives logarithmic dependence $r-1\eqsim \log((w)_{A_\infty} ^\text{dyadic}-1)^{-1}$ as $(w)_{A_\infty} ^\text{dyadic}\to 1^+$, where $r$ is the local integrability exponent in the reverse H\"older inequality.
\end{remark}

The rest of this paper is organized as follows. In \S\ref{s.aux} we provide some technical tools that will be used in the proofs of the main results. This section plays a purely supportive role. Section \S\ref{s.oned} contains all the one-dimensional results. More precisely, in \S\ref{s.onesided} we prove sharp reverse H\"older inequalities for one-sided weights in $A_1 ^+$ and $A_\infty ^+$. In \S\ref{s.twosidedAinfty} we prove the two-sided sharp reverse H\"older inequalities for $A_\infty$. In \S\ref{s.endpointoned} we provide the statements and proofs of the weak-type reverse H\"older inequalities at the endpoint, including the proof of Theorem~\ref{t.A1endpoint} for $A_1$ weights. We also give corresponding endpoint results for $A_\infty$ and for one-sided weights in $A_1 ^+$ and $A_\infty ^+$. In section \S\ref{s.genmeasuoned} we discuss some extensions of our results to classes of weights defined with respect to general non-atomic measures on the real line. In \S\ref{s.higher} we give the proof of Theorem~\ref{t.inftyn} and discuss extensions of our method of proof to more general contexts, as for example, to multiparameter weights defined with respect to product measures.
\section{Some auxiliary results}\label{s.aux}
\subsection {Flat \texorpdfstring{$A_\infty$}{Ainfty} weights} The main guiding principle behind the consideration of flat weights  is the fact that $[w]_{A_p}=1$ if and only if $w$ is a constant weight. Although we do not, strictly speaking, use this statement in the present paper we note, for completeness, that this characterization extends naturally to the $A_\infty$ case; namely, we have that $(w)_{A_\infty}=1$ if and only if the weight $w$ is constant. This is the content of the following simple proposition. 

\begin{proposition} We have that $(w)_{A\infty}=1$ if and only if there exists a constant $c> 0$ such that $w(x)=c$ for almost every $x\in \R^n.$	
\end{proposition}

\begin{proof} It is obvious that $(w)_{A_\infty}=1$ if $w$ is constant almost everywhere in $\R^n$. Now, assuming $(w)_{A_\infty}=1$ we have for any cube $Q\subset \R^n$ that
	\[
	\begin{split}
	0 &\leq \fint_Q |\M (w\ind_Q)(x)-w(x)|=\fint_Q \big(\M (w\ind_Q)(x) - w(x)\big)dx
	\\
	& = \fint_Q \big(\M (w\ind_Q)(x)-w_Q\big)dx=\fint_Q |\M (w\ind_Q)(x)-w_Q|dx \leq ((w)_{A_\infty}-1)w_Q.
	\end{split}
	\]
Thus if $(w)_{A_\infty}=1$ we get that for every cube $Q$ and almost every $x\in Q$ we have $w(x)=\M_Q(w\ind_Q)(x)=w_Q$. Now let $x,y\in\R^n$ and consider a cube $S\ni x,y$. We get that $w(x)=w_S=w(y)$ for almost every $x,y \in R^n$ and thus $w$ is constant almost everywhere in $\R^n$. Since we know that Muckenhoupt weights are non-zero almost everywhere we conclude that necessarily $c>0$.
\end{proof}

\subsection {The main lemma} Given some function $v:\R^n\to \R_+$ and some cube $Q$ in $\R^n$, our general approach to proving a reverse H\"older inequality for $v$ on $Q$ will be to establish estimates for the superlevel sets of $v$ in the form $v(\{x\in Q:\,v>\lambda\})\leq c \lambda |\{x\in Q:\,v>\lambda\}|$ for appropriate values of $\lambda,c>0$. This point of view is very well documented in the literature as for example in \cites{BSW,Kin}. The following lemma will be used repeatedly in the current paper and it is essentially identical to \cite{Kin}*{Lemma 2.14}. We include the easy proof for completeness.

\begin{lemma}\label{l.master} Let $v:\R^n\to \R$ be a non-negative, locally integrable function and $Q$ be a cube in $\R^n$. Suppose that there exists $\lambda_o\geq 0$ and $\delta>1$ such that
	\[
	v(x\in Q:\, v(x)>\lambda\})\leq \delta\lambda |\{x\in Q:\, v(x)>\lambda\}|,\qquad \lambda_o \leq \lambda<\infty.
	\]
Then for all $1\leq r <\delta/(\delta-1)$ we have
\[
\fint_Q v^r \leq \lambda_o ^{r-1} \frac{r'-1}{r'-\delta} \fint_Q v  .
\]
\end{lemma}

\begin{proof} Let us call $E_\lambda\coloneqq\{x\in Q:\, v(x)>\lambda\}$. Using the hypothesis we can then estimate
\[
\begin{split}
\int_{E_{\lambda_o}} v^r &= (r-1)\int_0 ^{\lambda_o} \lambda^{r-2}v(E_{\lambda_o})d\lambda+(r-1)\int_{\lambda_o} ^\infty \lambda^{r-2} v(E_\lambda)d\lambda	
\\
&\leq (r-1)\int_0 ^{\lambda_o} \lambda^{r-2}v(E_{\lambda_o})d\lambda+(r-1)\int_{\lambda_o} ^\infty \delta \lambda ^{r-1}|E_\lambda|d\lambda
\\
& \leq \lambda_o ^{r-1} v(E_{\lambda_o}) +\frac{\delta}{r'}\Big(\int_{E_{\lambda_o}} v^r - |E_{\lambda_o}|\lambda_o ^r  \Big).
\end{split}
\]
Now by replacing $v$ with $\min(v,N)$ for some large positive integer $N$ we can assume that $\int_K v^r<+\infty$ for any set $K$ of finite measure. Our estimates will not depend on $N$ so we this poses no restriction on $v$ as we can let $N\to +\infty$ at the end of the proof. With this observation in mind the previous estimate implies that for all $r<\delta/(\delta-1)$ we have
\[
\int_{E_{\lambda_o}}v^r\leq \frac{r'}{r'-\delta} \lambda_o ^{r-1} \big(\nu(E_{\lambda_o})-\frac{\delta}{r'}\lambda_o|E_{\lambda_o}|\big).
\]
The hypothesis for $\lambda=\lambda_o$ says that $\delta \lambda_o|E_{\lambda_o}|\geq  v(E_{\lambda_o})$ which, replacing in the estimate above, yields
\[
\int_{E_{\lambda_o}}v^r\leq \frac{r'/r}{r'-\delta} \lambda_o ^{r-1}  \nu(E_{\lambda_o})=\frac{r'-1}{r'-\delta}\lambda_o ^{r-1}\nu(E_{\lambda_o}).
\]
Then for the cube $Q$ we have
\[
\fint_Q v^r \leq \lambda_o ^{r-1}\frac{v(Q\setminus E_{\lambda_o})}{|Q|}  + \frac{r'-1}{r'-\delta}\lambda_o ^{r-1}\frac{\nu(E_{\lambda_o})}{|Q|}\leq \lambda_o ^{r-1}\frac{r'-1}{r'-\delta}\frac{\nu(Q)}{|Q|}
\]
as desired.
\end{proof}
The lemma above will allow us to reduce the proof of reverse H\"older inequalities to estimates of the form $v(E_{\lambda})\leq c(v) \lambda  |E_{\lambda}|$ for $\lambda\geq \lambda_o$. In higher dimensions we prove such estimates with $c(v)$  being asymptotically sharp as $(v)_{A_\infty}\to 1^+$. These estimates depend on the usual local Calder\'on-Zygmund decomposition in $Q$. However, in dimension one, we can prove the estimate required for Lemma~\ref{l.master} with the best possible choices of $c(v)$ and $\lambda_o$. These in turn depend on some covering lemmas that are specific to the topology of the real line and which we describe below.

\subsection {Covering lemmas on the real line}We remember here that, on the real line, $\M f $ denotes the non-centered Hardy-Littlewood maximal operator. We will also need the following one-sided versions $\M^+f,\M^-f$, defined for $f\in L^1 _{\text{loc}}(\R)$ as
\[
\M^+f(x)\coloneqq \sup_{h>0} \int_x ^{x+h} |f(t)|dt,\qquad \M^-f(x)\coloneqq \sup_{h>0} \int_{x-h} ^x |f(t)|dt,\qquad x\in\R.
\]
The following version of the rising sun lemma is essentially taken from \cite{S} and it will be essential in understanding the structure of sets of the form 
$\{t\in\R:\,\M^-(v\ind_I)(t)>\lambda\}$. An obvious symmetric version can be written for $\M^+$.
\begin{lemma}\label{l.risingsun} Suppose that $v\geq 0$ is a locally integrable function and let $I$ be a finite interval of the real line. Then the set 
$ \{x\in\R:\,M^-(v\ind_I) >\lambda\}=\cup_j(a_j,b_j)$  where the intervals $(a_j,b_j)$ are pairwise disjoint and for every $j$ and $x\in (a_j,b_j)$ we have
	\[
	\fint_{a_j} ^{x} v\ind_I \geq \lambda.
	\] 
In particular, $v((a_j,b_j)\cap I)=\lambda |(a_j,b_j)|$ for every $j$.	
\end{lemma}
We formulate now a two-sided version of the rising-sun lemma which will be used to analyze sets of the form $\{t\in\R:\,  \M(v\ind_I)(t)>\lambda \}$.
\begin{lemma}\label{l.risingsun2}
Let $f\geq 0$ be an integrable function with compact support and let $\lambda>0$. There exists a countable collection of pairwise disjoint bounded intervals $\{(a_j,b_j)\}_j$ such that 
	\[
	E_\lambda\coloneqq \{x\in\R:\, \M f(x)>\lambda\}=\bigcup_j (a_j,b_j).
	\]
Furthermore we have the following properties
	\begin{itemize}
		\item [(i)] The $(a_j,b_j)$'s are maximal in the following sense: if $J$ is an interval with $\fint_J f>\lambda$ then there exists $j$ such that $J\subseteq (a_j,b_j)$. In particular, if $J$ is an interval and $J\cap  (a_j,b_j) ^\mathsf{c}\neq \emptyset$ for some $j$ then $\fint_J f\leq \lambda$.
		\item [(ii)] For every $j$ and every $a_j\leq x \leq b_j$ we have
		\[
		 \fint_{(a_j,x)} f \leq\lambda\qquad\text{and}\qquad \fint_{(x,b_j)}  f   \leq \lambda.
		\]
		\item[(iii)] The maximal function can be localized on the intervals $(a_j,b_j)$
		\[
		 (\M f) \ind_{(a_j,b_j)}= \M(f\ind_{(a_j,b_j)}) \ind_{(a_j,b_j)}.
		\]
	\end{itemize}
\end{lemma}

\begin{proof} The set $E_\lambda$ is open thus it can be written as a countable union of pairwise disjoint open intervals. Note that the intervals $(a_j,b_j)$ are the connected components of $E_\lambda$. Also, $E_\lambda$ is a bounded set since $f$ has finite integral and is zero outside a compact set and thus each $(a_j,b_j)$ is a bounded interval.

To see (i) suppose that for some $J$ we have $\fint_J f>\lambda$. Then $J\subseteq E_\lambda$ and thus there exists at least one $j$ such that $J\cap(a_j,b_j)\neq \emptyset$. Then $J\cup(a_j,b_j)$ is an open interval contained in  $E_\lambda$ so it must be contained in one of the connected components of $E_\lambda$. Since $J$ intersects $(a_j,b_j)$ and the $(a_j,b_j)$'s are disjoint, we get the claim. The second conclusion follows immediately by the maximality just proved.

For (ii) take any $x\in (a_j,b_j]$ and some interval $J\coloneqq (a_j-\varepsilon,x+\varepsilon)$ for $\varepsilon>0$. As $J\cap(a_j,b_j)^{\mathsf{c}}\neq \emptyset$, we get by (i) that $\fint_J f\leq \lambda$. Letting $\varepsilon\to 0^+$ we can conclude that $\fint_{(a_j,x)}f\leq \lambda$. A similar argument gives the estimate for intervals $(x,b_j)$.

Finally (iii) follows by the maximality proved in (i). Indeed, let $x\in(a_j,b_j)$ for some $j$ and let $J\ni x$ be some interval. If $J\not\subseteq (a_j,b_j)$ then we have $\fint_J f\leq \lambda$ thus
\[
\sup_{\substack{x\in J\\ J\cap(a_j,b_j)^\mathsf{c}\neq \emptyset}}\fint_J f \leq \lambda.
\]
On the other hand since $x\in(a_j,b_j)$ there exists $J'\ni x$ with $\fint_{J'} f> \lambda$ and $J'\subseteq (a_j,b_j)$ by (i). Thus 
\[
\sup_{\substack{ J\subseteq (a_j,b_j) \\ J\ni x}}\fint_J f > \sup_{\substack{x\in J\\ J\cap(a_j,b_j)^\mathsf{c}\neq \emptyset}}\fint_J f
\]
which gives the claim in (iii).	
\end{proof}

\section{Sharp reverse H\"older inequalities in one dimension.}\label{s.oned}  In this section we focus on dimension $n=1$ where, typically, sharp reverse H\"older inequalities are more readily available, their proofs being based on the special topology of the real line. Let us consider the reverse H\"older inequality in the following form; for some locally integrable, non-negative function $w$ there exists $r_w>1$ and a function $[1,r_w)\ni r\mapsto C_{w}(r)$ such that
\[
\frac{1}{|I|}\int_I w^r \leq C_w(r) \Big(\frac{1}{|I|}\int_I w \Big)^r
\]
for all $1\leq r < r_w$. The point is that $C_w(r)<+\infty$ for $1\leq r<r_w$ but, in a sharp version of the inequality above $C_w(r)$ blows up as $r\to r_w ^-$. 

Ideally, one seeks such reverse H\"older inequalities which are sharp, both in terms of the upper bound for the integrability exponent $r_w>1$ as well as in terms of the multiplicative constant $C_w(r)$. As was discussed in the introduction the sharp reverse H\"older inequality for $A_1$ weights on the real line was proved by Bojarski, Sbordone, and Wik, in \cite{BSW}: If $w$ is an $A_1$ weight on the real line with constant $[w]_{A_1}$ then
\[
\frac{1}{|I|}\int_I w^r \leq \frac{1}{[w]_{A_1} ^{p-1}}\frac{p'-1}{p'-[w]_{A_1}} \Big(\frac{1}{|I|}\int_I w \Big)^r
\]
for all $r\in [1,r_w)$, where $r_w=\frac{[w]_{A_1}}{[w]_{A_1}-1}$. The inequality above can be seen as a sharp embedding of the class $A_1$ into the reverse H\"older classes.

In this section we first investigate the case of one-sided Muckenhoupt weights, which shares a lot of the ideas and techniques with the two-sided case. In fact, although the corresponding reverse H\"older inequalities are weaker, and this is necessarily so since one-sided weight are in general a lot worse behaved that two-sided weights, they are very similar formally. Furthermore, one can typically deduce reverse H\"older inequalities for two-sided weights by studying the reverse H\"older inequalities for forward and backward Muckenhoupt weights and glueing them together. This is however a particular feature of the one-dimensional case as in higher dimensions the theory of one-sided weights is not as complete, nor is it clear which geometric setup, if any, would give analogous results.

\subsection{Reverse H\"older inequalities for one-sided weights}\label{s.onesided}  A weight $w$ on the real line is said to be in the forward class $A_p ^+$ if 
\[
[w]_{A_p ^+}\coloneqq \sup_{a<b<c}\frac{w(a,b)}{|(a,c)|}\Big(\frac{\sigma(b,c)}{|(a,c)|}\Big)^{p-1}<+\infty
\]
where $\sigma\coloneqq w^{-\frac{1}{p-1}}$ is the dual weight of $w$. One-sided weights have been studied extensively, see for example \cites{CUNO,MReyes,MPT,MOT,MT,S}, and it is well known that $w\in A_p ^+$ if and only if the forward maximal function
\[
\M^+f(x)\coloneqq \sup_{h>0} \int_x ^{x+h} |f(t)|dt
\]
maps $L^p(w)$ to itself. This again happens if and only if $\M^+$ maps $L^p(w)$ to $L^{p,\infty}(w)$. For $p=1$ the class $A_1 ^+$ is defined as the class of weights $w$ for which
\[
[w]_{A_1 ^+}\coloneqq \NOrm \frac{\M^- w}{w}.L^\infty(\R).<+\infty
\]
where we remember that $\M^-$ is the backwards Hardy-Littlewood maximal function
\[
\M^- f(x)\coloneqq \sup_{h>0} \frac{1}{h} \int_{x-h} ^x |f(t)|dt.
\]
Finally, the class $A_\infty ^+$ can be defined in several equivalent ways. For example we have $A_\infty ^+ =\cup{_{p>1}} A_p ^+$. Alternatively $A_\infty ^+$ can be defined as the class of weights for which
\[
(w)_{A_\infty ^+}\coloneqq \sup_I \frac{1}{w(I)}\int_I \M^-(w\ind_I)<+\infty,
\]
where the supremum is taken over all finite intervals $I\subset \R$. For this definition see \cite{MT}. Completely symmetric definitions can be given for the classes $A_p ^-$ for $1\leq p \leq \infty$ and all the results that we prove below have symmetric versions for $A_p ^-$.

Our first task is to state and prove the analog of the Bojarski-Sbordone-Wik theorem for $A_1 ^+$ weights. We note here that almost sharp reverse H\"older inequalities have been proved for example in \cite{MT}. However the precise statement below, which is the best possible reverse H\"older inequality for $A_1 ^+$ weights, appears to be new. The proof of this inequality follows more or less well known arguments, as for example in \cite{Kin}, and in particular it follows by establishing estimates as in the hypothesis of Lemma~\ref{l.master}.

\begin{theorem}\label{t.onesidedA1} Let $w\in A_1 ^+$ and suppose that $[w]_{A_1 ^+}\leq \delta$ for some $\delta\geq 1$. Then for all $1\leq r< \delta/(\delta-1)$ and for all finite intervals $(a,b)$ we have
\[
\int_{(a,b)} w^r \leq \frac{1}{\delta^{r-1}} \frac{r'-1}{r'-\delta}  \M^-(w\ind_{(a,b)})(b)^{r-1} w(a,b).
\]
It follows that for all $a<b<c$ and all $r<\delta/(\delta-1)$ we have
\[
|(b,c)|^{r-1}\int_{(a,b)} w^r \leq \frac{1}{\delta^{r-1}}\frac{r'-1}{r'-\delta} w(a,c)^r.
\]
Furthermore, these reverse H\"older inequalities are best possible, both in terms of the range of integrability as well as in terms of the multiplicative constant.
\end{theorem}

\begin{proof} Let $I=(a,b)$ be an interval and consider the set $E_\lambda\coloneqq \{x\in I:\,w>\lambda\}$. Set also $\lambda_o\coloneqq \M^-(w\ind_I)(b)/\delta$. The first reverse H\"older inequality of the theorem will follow from Lemma~\ref{l.master} once we show that $w(E_\lambda)\leq \delta \lambda|E_\lambda|$ for all $\lambda\geq \lambda_o$.
	
To that end let us first assume that $w$ is continuous. Then the set $\{x\in\R:\, w>\lambda\}$ is open thus there exists a collection of disjoint intervals $\{I_j\}_j= \{(a_j,b_j)\}_j$ such that $E_\lambda =\cup_j I_j \cap I$. For each interval $I_j$ there are two possibilities. If $b_j< b$ then by the definition of $[w]_{A_1 ^+}$ we will have $w(I_j\cap I)\leq [w]_{A_1 ^+} w(b_j)|I_j\cap I|\leq \delta \lambda |I_j\cap I|$ as $b_j\notin E_{\lambda}$. If $b \leq b_j$ then 
	\[
w(I_j\cap I)= \frac{w (I_j\cap I)}{\delta| I_j\cap I|} \delta|I_j\cap I|\leq \frac{\M^-(w\ind_I)(b)}{\delta}\delta|I_j\cap I| \leq \delta\lambda |I_j\cap I|
	\]
for $\lambda\geq \lambda_o$. Combining these observations we get that $w(E_\lambda )\leq \delta\lambda |E_\lambda |$ for all $\lambda\geq \lambda_o$. 

For general $w\in A_1 ^+$ we use a standard approximation argument as in \cite{Kin1}. Indeed, let $\phi \in C^\infty_c(\R)$ be a smooth function with compact support satisfying $\phi\geq 0$ and $\int \phi =1$. Denoting by $\phi_t(x)\coloneqq t^{-1}\phi(x/t)$ we readily see that $\phi_t*w$ is an $A_1 ^+$ weight, uniformly in $t$, with $[w*\phi_t]_{A_1 ^+}\leq \delta$. Then for each $t>0$ and $\lambda\geq \lambda_o$ we have
\[
\int_{E_\lambda } w*\phi_t  \leq \delta \lambda |E_\lambda |
\]
and the conclusion follows for general $w\in A_1 ^+$ by letting $t\to 0^+$.

The second reverse H\"older inequality follows from the first via a standard argument. Indeed for all $a<b<x<c$ we write
\[
\begin{split}
	\int_{(a,b)} w^r \leq \int_{(a,x)}w^r \leq C_w (r) \M ^-(w\ind_{(a,c)})(x)^{r-1} w(a,c)
\end{split}
\]
where we have set $C_w(r)\coloneqq  \delta^{1-r} (r'-1)/ (r'-\delta)$. Then 
\[
(b,c)\subseteq \bigg\{x\in\R:\, \M^-(w\ind_{(a,c)}(x)) ^{r-1}> w(a,c)^{-1} C_w(r) ^{-1}\int_{(a,b)}w^r \bigg\}
\]
and the second reverse H\"older inequality of the lemma follows by the weak $(1,1)$ type of $\M^-$ with constant $1.$

For the sharpness claim, it is enough to check that the second reverse H\"older inequality is best possible, as it follows from the first. For this we use the well known example $w(x)\coloneqq |x|^{\tau-1}$ for $0<\tau<1$. We have that $w\in A_1$ and thus $w\in A_1 ^+$ while one easily verifies that $[w]_{A_1 ^+}= \tau^{-1}$. We consider the points $0<1-\varepsilon<1$ for $\varepsilon\in(0,1)$. Then, assuming a reverse H\"older inequality of the form
\[
|(b,c)|^{r-1}\int_{(a,b)} w^r \leq C_w(r) w(a,c)^r
\]
for all $a<b<c$, and plugging in the previous choices of $a,b,c$, and $w$, we get that $\int_{(0,1-\varepsilon)}w^r<+\infty$ exactly when $r<\frac{1/\tau}{1-1/\tau}$. Furthermore we must have
\[
C_w(r)\geq \tau^r ((\tau-1)r+1)^{-1} \sup_{0<\varepsilon<1}\varepsilon^{r-1}  [1-(1-\varepsilon)^{(\tau-1)r+1}] =\frac{1}{(1/\tau)^{r-1}} \frac{r'-1}{r'-\tau^{-1}}
\]
which is exactly the constant in the reverse H\"older inequality of the theorem.
\end{proof}

\begin{remark}\label{r.ainftya1}In what follows we will state and prove a sharp one-sided reverse H\"older inequality for weights in $A_\infty ^+$. Observe that the $A_\infty ^+$ condition can be written in the form
\[
  \frac{\M^- (w\ind_I)(I)}{|I|}  \leq (w)_{A_\infty ^+} \frac{w(I)}{|I|}\leq (w)_{A_\infty ^+} \M^-(w\ind_I)(b)
\]
for all intervals $I=(a,b)\subset \R$. As in the comments preceding the statement of Theorem~\ref{t.inftyn} the heuristic is that the function $\M^-(w\ind_I)$ behaves locally like an $A_1 ^+$ weight. This remark should of course be taken with a grain of salt because of the localization of the maximal function in the definition of the $A_\infty ^+$ constant.
\end{remark}

In order to state the reverse H\"older inequality for $A_\infty ^+$ in its sharpest possible form we need a technical definition.
\begin{definition} Let $f$ be an integrable function with compact support. We then define
	\[
	\M^- _{[2]}f(x) \coloneqq \sup_{h>0} \frac{1}{h} \int_{x-h} ^x \M^-(f\ind_{(x-h,x)})(y)dy,\quad x\in \R.
	\]
\end{definition}
The rather involved maximal operator in the definition above can be understood as a local version of the second iteration of $\M^-$. The motivation for this definition is essentially contained in Remark~\ref{r.ainftya1} as we need to find an appropriate $A_\infty$-analog of the value $\M^-(w\ind_I)(b)$. Recall that the last quantity played an important role in the statement and proof of the sharp reverse H\"older inequality for $A_1 ^+$ weights.  

The following estimate is an easy consequence of the definitions.
\begin{proposition}\label{p.doubleM-} If $I=(a,b)$ is a finite interval then for all $w\in A_\infty ^+$ with $(w)_{A_\infty ^+}\leq \delta$ we have
	\[
    \frac{1}{|I|}\M^-(w\ind_I)(I) \leq  \M^- _{[2]}(w\ind_I)(b) \leq \delta \M^-(w\ind_I)(b) .
	\]
\end{proposition}
We are now ready to prove the sharp reverse H\"older inequality for $A_\infty ^+$.
\begin{theorem}\label{t.onesided} Let $w$ be a weight in $A_\infty ^+$, with constant $(w)_{A_\infty ^+}\leq \delta $ for some $\delta\geq 1$. Then for all $1\leq r<\delta/(1-\delta)$ and for all finite intervals $(a ,b)$ we have  
\[
\int_{(a,b)} \M^-(w\ind_{(a,b)})^r \leq \frac{1}{\delta^{r-1}} \frac{r'-1}{r'-\delta}\M^- _{[2]}  (w\ind_{(a,b)})(b)^{r-1} \M^-(w\ind_{(a,b)})(a,b)
\]	

\end{theorem}

\begin{proof}
We fix some interval $I\coloneqq (a,b)$ and write $E_\lambda\coloneqq\{x\in I: \, \M^-(w\ind_I)>\lambda\}=I\cap \cup_j I_j = I\cap \cup_j (a_j,b_j)$ by Lemma~\ref{l.risingsun}. Note that, because of Lemma~\ref{l.risingsun}, we have that $I_j \cap I\neq \emptyset$ while $a_j \geq a$  for all $j$ from the definition of $\M^-$. Another important observation is that for $x\in I\cap I_j$ the maximal function $\M^-(w\ind_{I})(x)$ can be localized on the intervals $I\cap I_j$, namely we have 
\[
\M^-(w\ind_I)(x)=\M^-(w\ind_{I\cap I_j})(x),\qquad  x\in I_j.
\]
To see this suppose that $x\in I\cap I_j$ and consider $u\leq a_j$. Then
\[
\fint_u ^x w\ind_I= \frac{a_j-u}{x-u}\fint_u ^{a_j}w\ind_I+\frac{x-a_j}{x-u}\fint_{a_j} ^xw\ind_I.
\]
However, since $a_j\notin E_\lambda$ we have $\fint_{u} ^{a_j} w\ind_I\leq \lambda\leq \fint_{a_j} ^x w\ind_I$ by Lemma~\ref{l.risingsun}. This shows that
\[
\fint_{u} ^x w\ind_I \leq \fint_{a_j} ^x w\ind_I= \fint_{a_j} ^x w\ind_{I\cap I_j} \leq \M^- (w\ind_{I\cap I_j})(x)
\]
and thus $\M^- (w\ind_I) (x)=\M^-w(\ind_{I\cap I_j})(x)$ for $x\in I\cap I_j$. 

We now define $\lambda_o \coloneqq  \delta^{-1}  \M^- _{[2]}(w\ind_I)(b)$ and we show the estimate $\M^-(w\ind_I)(E_\lambda)\leq \delta \lambda |E_\lambda|$ for all $\lambda\geq \lambda_o$. To see this note that for the intervals $I_j$ there are two possibilities. If $b_j \leq b$ then $I_j\subseteq I$ and by the definition of $(w)_{A_\infty ^+}$ and the previous observations we have
\[
\M^-(w\ind_I)(I\cap I_j)=\int_{I_j}\M^-(w\ind_{ I_j})\leq \delta w( I_j)=  \delta \lambda |  I_j| =\delta\lambda|I\cap I_j|.
\]
The other case is that $b_j>b$ for some interval $I_j$. Then Proposition~\ref{p.doubleM-} implies that for $\lambda\geq \lambda_o$ we have
\[
\M^-(w\ind_I)(I\cap I_j)=\M^-(w\ind_{I\cap I_j})(I\cap I_j)\leq |I\cap I_j|\; \M^- _{[2]}(w\ind_{I})(b)  \leq \delta \lambda_o|I\cap I_j| \leq \delta\lambda|I\cap I_j|.
\]
Combining these two cases we see that $\M^-(w\ind_I)(E_{\lambda})\leq \lambda\delta |E_\lambda|$ whenever $\lambda\geq \lambda_o$ and the estimate of theorem follows by Lemma~\ref{l.master}.

In order to see that the range of integrability in the reverse H\"older inequality is best possible we use again the weight $w(x)\coloneqq |x|^{\tau-1}$ for $\tau\in(0,1)$ which satisfies $[w]_{A_1 ^+}= \tau^{-1}$. Taking $I\coloneqq (0,1)$ we now have that $\M^-(w\ind_I)(x)=\tau ^{-1} w\ind_I$ and also $M^- _{[2]}(w\ind_I)(x)=\tau^{-2}w\ind_I$. A simple calculation shows that this example is an extremizer for our inequality.
\end{proof}
In the theorem above we have chosen to write the reverse H\"older inequality for $A_\infty ^+$ in a rather involved way. Our reason for doing so was to highlight the similarity between the reverse H\"older inequalities for $A_1 ^+$ and $A_\infty^+$ and to show that the intuition given by Remark~\ref{r.ainftya1} can be made rigorous. Below we deduce some more practical versions of the reverse H\"older inequality for $A_\infty ^+$.

\begin{corollary} Let $w\in A_\infty ^+$ with $(w)_{A_\infty ^+}\leq \delta$ for some $\delta\geq 1$. Then for all $r<\delta/(\delta-1)$ and for all finite intervals $(a,b)$ we have
 	\[
	\begin{split}
 \int_{(a,b)} \M^-(w\ind_{(a,b)}) ^r & \leq \delta  \frac{r'-1}{r'-\delta}  \M^-(w\ind_{(a,b)})(b)^{r-1} w(a,b) .
 \end{split}
	\]
Furthermore, for all real numbers $a<b<c$ and for all $r<\delta/(\delta-1)$ we have
\[
|(b,c)|^{r-1} \int_{(a,b)} \M^-(w\ind_{(a,b)})^r \leq \delta \frac{r'-1}{r'-\delta} w(a,c)^r.
\]
These inequalities are best possible, both in terms of the range of integrability, as well as in terms of the multiplicative constant.	
\end{corollary}

\begin{proof} To prove the first inequality we just use Theorem~\ref{t.onesided} together with the remark of Proposition~\ref{p.doubleM-}. The second inequality together with the optimality claim follow in exactly the same way as in the end of the proof of Theorem~\ref{t.onesidedA1} and the details are left to the reader.
\end{proof}

\subsection{Two-sided reverse H\"older inequalities in one-dimension}\label{s.twosidedAinfty}
As we have already commented, the two-sided version of Theorem~\ref{t.onesidedA1}, namely the sharp reverse H\"older inequality for $A_1$ weights goes back to Bojarski, Sbordone, and Wik, from \cite{BSW}. The main goal of this paragraph is to prove the $A_\infty$ version of this result. We rely again on the heuristic described in Remark~\ref{r.ainftya1}, namely, if $w\in A_\infty$ then $\M(w\ind_I)$ behaves locally like an $A_1$ weight. As in the one sided case the proof relies on an appropriate version of the rising sun lemma, which is contained in Lemma~\ref{l.risingsun2}.

We now give the  proof of the sharp reverse H\"older inequality for weights in $A_\infty$.

\begin{proof}[Proof of Theorem~\ref{t.RHI-1d}] Given some open interval $I=(a,b)$ and $w\in A_\infty$ let $E_\lambda\coloneqq\{x\in I:\, \M(w\ind_I)>\lambda\}$. We set $\lambda_o\coloneqq  \M(w\ind_I)(I) ( \delta|I|)^{-1}$. We first show that for $\lambda\geq \lambda_o$ we have $\M(w\ind_I)(E_\lambda)\leq \delta \lambda |E_\lambda|$.
	
The set $\{x\in\R:\, \M(w\ind_I)>\lambda\}$ is open thus by Lemma~\ref{l.risingsun2} there exists a countable collection $\{I_j\}_j=\{(a_j,b_j)\}_j$ of pairwise disjoint intervals and $E_\lambda = \cup_j I_j\cap I$. If $|I\cap I_j|<|I|$ then $I\cap I_j \subsetneq I$ and it is of the form $(a_j,b)$, with $a_j> a$, or $(a,b_j)$, with $b_j<b$, or $I\cap I_j=I_j\subsetneq I$. Using (ii) and (iii) of Lemma~\ref{l.risingsun2} we have for these intervals that
\[
\M(w\ind_I)(I\cap I_j)= \M(w\ind_{I\cap I_j})(I\cap I_j)\leq \delta  \frac{w(I_j\cap I)}{|I_j\cap I|}|I_j\cap I|\leq \delta \lambda|I_j\cap I|.
\]
It remains to consider the case $|I_j\cap I|=|I|$. However in this case we have
\[
\M(w\ind_{I})(I_j\cap I) =\int_I \M(w\ind_I)=\frac{\M(w\ind_I)(I)}{\delta |I|}\delta|I|=\lambda_o \delta |I|\leq \lambda \delta|I|
\]
since $\lambda\geq \lambda_o$.

Thus $\M(w\ind_I)(E_\lambda)=\sum_j \M(w\ind)(I_j\cap I)\leq \delta \lambda \sum_j|I_j\cap I|=\delta\lambda$ for all $\lambda\geq \lambda_o$ and the reverse H\"older inequality of the theorem follows by an application of Lemma~\ref{l.master}.

The optimality of the exponent in the reverse H\"older inequality and of the multiplicative constant follows by considering the usual example $w(x)\coloneqq |x|^{\tau-1}$ for $0<\tau<1$ on the interval $(0,1)$.
\end{proof}
The reverse H\"older inequality was stated above for the maximal function $\M(w\ind_I)$ instead of $w$ itself, as we think that this is a natural way to write sharp reverse H\"older inequalities for $A_\infty$ involving the Fujii-Wilson constant. However we record the following easy corollary.

\begin{corollary} Let $w\in A_\infty$ with constant $(w)_{A_\infty}\leq \delta$ for some $\delta\geq 1$. Then for all $1\leq r<\delta/(\delta-1)$ and all finite intervals $I$ we have
	\[
	\fint_I w^r \leq \delta \frac{r'-1}{r-\delta} \big(\fint_I w \big)^r
	\]
and the range of integrability is best possible.	
\end{corollary}

\subsection{Sharp weak-type reverse H\"older inequalities at the endpoint}\label{s.endpointoned} The reverse H\"older inequality for $A_1$ weights on the real line has the form
\[
\fint_I w^r \leq \frac{1}{\delta^{r-1}}\frac{r'-1}{r'-\delta}\big(\fint_I w\big)^r
\]
whenever $1\leq r<r_w\coloneqq [w]_{A_1}/([w]_{A_1}-1)$. We know that this range of integrability cannot be improved as can be easily seen by inspecting the example $w(x)\coloneqq |x|^{\tau-1}$ for $\tau\in(0,1)$. Indeed, $[w]_{A_1}=\tau^{-1}$ and $w$ is not locally in $L^{r_w}$. However, it can be easily checked that $w\in L^{r_w,\infty}(I)$ for any interval $I$. This suggests that we might be able to consider weak-type reverse H\"older inequalities at the endpoint, that is, inequalities of the form
\[
 \frac{1}{|I|}|\{t\in\R:\, w>\lambda\}|  \leq \frac{c_w}{\lambda^{r_w}}  \big(\fint_I w\big)^{r_w}.
\]
Indeed, such endpoint estimates are proved for example in \cite{Ose} in the case that $w$ is a dyadic $A_1$ weight. See also \cite{Rey} for a Bellman function approach.

We provide here the analog of the result in \cite{Ose} for non-dyadic weights in $A_1$, together with the corresponding endpoint reverse H\"older inequality for $A_\infty$ weights on the real line. We don't pursue higher dimensional versions of these inequalities as, in dimensions greater than $1$, it is not clear what the appropriate endpoint should be. 

In what follows, for every finite interval $I$ on the real line and a weight $w$ we define
\[
\|w\|_{L^{r,\infty}(I,\frac{dx}{|I|})}\coloneqq  \frac{1}{|I|^\frac{1}{r}}\sup_{\lambda>0}\lambda|\{x\in I:\, w(x)>\lambda\}|^\frac{1}{r}.
\]
Furthermore, for a non-negative function $w$ restricted to some interval $I$ we will write $w^*$ for the non-increasing rearrangement of $w\ind_I$. Note that we use the convention that $w^*$ is \emph{left-continuous}, as in \cite{BSW}. With these definitions in hand we can now give the proof of Theorem~\ref{t.A1endpoint}.
\begin{proof}[Proof of Theorem~\ref{t.A1endpoint}] We begin by fixing an interval $I$ and a weight $w$ with $[w]_{A_1}\leq \delta$. From now on we can assume that $w=w\ind_I$ is supported on $I$. As it is shown in \cite{BSW} we then have that $w^*$ is in $A_1(0,|I|)$ with the same constant $\delta$. In particular we have for all $t\in (0,|I|]$ that
	\[	
\frac{1}{t}\int_0 ^t w^* (s)ds \leq \delta \essinf_{(0,|I|)}w^*=\delta w^*(|I|).
	\]
Note here that, as it was shown in \cite{BSW}, the fact that the rearrangement of $w$ has the same $A_1$-constant is specific to the one-dimensional case. It now follows from \cite{Wik}*{Lemma 2} that for all $t\in(0,|I|]$ we have
\[
\int_0 ^t w^* (s)ds \leq \big(\frac{t}{|I|}\big)^\frac{1}{\delta}\int_0 ^{|I|} w^* (s)ds
\]
Letting $E_\lambda\coloneqq \{I:\, w>\lambda\}$ we have
\[
w(E_\lambda)\leq \int_0 ^{|E_\lambda|}w^*(s)ds \leq \big(\frac{|E_\lambda|}{|I|}\big)^\frac{1}{\delta}	\int_0 ^{|I|}w^*(s)ds =\big(\frac{|E_\lambda|}{|I|}\big)^\frac{1}{\delta} w(I).
\]
As we trivially have $w(E_\lambda)>\lambda|E_\lambda|$ this implies that
\[
 |I|^\frac{1}{\delta} \lambda |E_\lambda|^{1-\frac{1}{\delta}}\leq w(I)
\]
which easily gives the desired estimate.

Now assume that the first conclusion of the theorem holds with some constant $C$ in place of $1$. Then for constant weights $w=\mu>0$ we have for every $\lambda<\mu$ that
\[
C\geq \frac{\lambda |E_\lambda|^\frac{\delta-1}{\delta}}{w_I|I|^\frac{\delta-1}{\delta}}=\frac{\lambda}{\mu}
\]
whence the optimality of the constant follows by letting $\lambda\nearrow \mu$. On the other hand, as we already discussed, the weight $w(x)\coloneqq |x|^{\tau-1}$ for $\tau\in(0,1)$ satisfies $[w]_{A_1} =\tau^{-1}$ and $w\notin L^{r,\infty}(0,1)$ for any $r>[w]_{A_1}/([w]_{A_1}-1)$.
\end{proof}

In order to prove the corresponding result for $A_\infty$ weights on the real line we will need to study how the $A_\infty$ constant behaves under rearrangements. The following lemma is analogous to \cite{BSW}*{Theorem 1}.
\begin{lemma}\label{l.rearinfty}Let $w\in A_\infty$ with constant $[w]_{A_\infty}\leq \delta$, for some $\delta>1$, and let $\M(w\ind_I)^*$ denote the non-increasing (left continuous) rearrangement of $\M(w\ind_I)\ind_I$, where $I$ is a bounded interval on the real line. Then for all $t\in(0,|I|]$ we have
\[
\frac{1}{t}\int_0 ^t \M(w\ind_I)^*(s)ds\leq \delta \M(w\ind_I)^*(t).	
\]
\end{lemma}

\begin{proof} We set $E_\lambda\coloneqq\{x\in I:\, \M(w\ind_I)(x)>\lambda\}$. From the proof of Theorem~\ref{t.RHI-1d} we have that $\M(w\ind_I)(E_\lambda)\leq \delta\lambda|E_\lambda|$ if $\lambda\geq \lambda_o\coloneqq \M(w\ind_I)(I)/(\delta|I|)$. Thus we have
	\[
	\frac{1}{|E_\lambda|}\int_0 ^{|E_\lambda|}\M(w\ind_I)^* =\frac{1}{|E_\lambda|}\int_{E_\lambda} \M(w\ind_I)\leq \delta \lambda \leq \delta \essinf_{t\in (0,|E_\lambda|)} \M(w\ind_I)^*(t) = \delta \M(w\ind_I)^*(|E_\lambda|)
	\]
since $\M(w\ind_I)^*$ is left continuous.  This proves the desired estimate for $t=|E_\lambda|$ for $\lambda\geq \lambda_o$. Now suppose that $\lambda<\lambda_o$. By the definition of $(w)_{A_\infty}$ we then have that $\lambda<w(I)/|I|$. Since $\M(w\ind_I)\geq w(I)/|I|$ throughout $I$ we get that in this case $E_\lambda=I$. Then we trivially have
\[
\begin{split}
\fint_0 ^{|I|} \M(w\ind_I)^*&=\fint_I \M(w\ind_I) \leq \delta \frac{w(I)}{|I|}\leq \delta \inf_I \M(w\ind_I)
\\
&=\delta\essinf_{\tau\in (0,|I|)}\M(w\ind_I)^*(\tau)=\delta\M(w\ind_I)^*(|I|)
\end{split}
\]
as $\M(w\ind_I)^*$ is left continuous. This shows the desired estimate whenever $t=|E_\lambda|$ for some $\lambda>0$.

Now for arbitrary $t\in(0,|I|]$ we define $\lambda_1>0$ and $t\in (0,|I|]$ by setting $\lambda_1\coloneqq \M(w\ind_I)^*(t)$ and $t_1\coloneqq \min \{\tau \in(0,t]:\, \M(w\ind_I)(\tau)=\lambda_1\}$. With these definitions we have that $|E_{\lambda_1}|=t_1\leq t$. Thus
\[
\begin{split}
\fint_0 ^t \M(w\ind_I)^* &= \frac{t_1}{t}\fint_0 ^{t_1}\M(w\ind_I)^*+\frac{t-t_1}{t}\fint_{t_1} ^t\M(w\ind_I)^*
\\
& \leq \frac{t_1}{t} \delta \M(w\ind_I)^*(t_1)+\frac{t-t_1}{t}  \lambda_1\leq \delta\lambda_1 
\\
&=\delta \essinf_{\tau\in(0,t)}\M(w\ind_I)^*(\tau)=\delta\M(w\ind_I)^*(t)
\end{split}
\]
which is the desired estimate.
\end{proof}
With the estimate of Lemma~\ref{l.rearinfty} in hand it is now easy to give the endpoint reverse H\"older inequality of $A_\infty$. We omit the details of the proof as it is essentially identical to that of Theorem~\ref{t.A1endpoint}.

\begin{theorem}\label{t.ainftyendpoint}Let $w\in A_\infty$ with $[w]_{A_\infty}\leq \delta$ for some $\delta>1$. Let $r_w\coloneqq \delta/(\delta-1)$. Then for every bounded interval $I$ we have that $w\in L^{r_w,\infty}(I,\frac{dx}{|I|})$ and
	\[
	 \| \M(w\ind_I)\|_{L^{r_w,\infty}(I,\frac{dx}{|I|})}    \leq   \fint_I \M(w\ind_I) .
	\]
Furthermore, the range of weak-type integrability and the multiplicative constant $1$ are best possible.
\end{theorem}

As a corollary we get some useful sharp estimates for $A_1$ and $A_\infty$ weights. These should be compared to the results and discussion in \cite{Rey}.
\begin{corollary} Let $w$ be a weight on the real line.
	\begin{enumerate}
		\item [(i)] If $w\in A_\infty$ then for all bounded intervals $I$ and all measurable sets $E\subseteq I$ we have
		\[
		\frac{w(E)}{w(I)}\leq [w]_{A_1} \big(\frac{|E|}{|I|}\big)^\frac{1}{[w]_{A_1}}.
		\]
		\item[(ii)] If $w\in A_\infty$ then for all bounded intervals $I$ and all measurable sets $E\subseteq I$ we have
		\[
		\frac{\M(w\ind_I)(E)}{\M(w\ind_I)(I)}\leq (w)_{A_\infty}\big(\frac{|E|}{|I|}\big)^\frac{1}{(w)_{A_\infty}}.
		\]
	\end{enumerate}
\end{corollary}

\begin{proof}To prove the estimate in (i) we note that by Theorem~\ref{t.A1endpoint} we have that $\|w\|_{L^{r_w,\infty}(I,dx/|I|)}\leq w_I$, where $r_w\coloneqq \delta/(\delta-1)$. Using H\"older's inequality for Lorentz spaces we get
	\[
	\frac{1}{|I|}\int w\ind_E \leq r_w ' \|w\|_{L^{r_w,\infty}(I,\frac{dx}{|I|})}\big(\frac{|E|}{|I|}\big)^\frac{1}{r_w '}\leq r_w ' w_I \big(\frac{|E|}{|I|}\big)^\frac{1}{r_w '}.
	\]
As $r_w ' =[w]_{A_1}$ we get the claim of (i). The proof of (ii) follows similarly, by using the estimate in Theorem~\ref{t.ainftyendpoint}.
\end{proof}
One can formulate similar endpoint reverse H\"older inequalities for one sided weights. Indeed, with only minor modifications in the proof one can get the following result.
\begin{theorem} Let $\delta\geq 1$ and set $r_w\coloneqq \delta/(\delta-1)$. If $w\in A_1 ^+$ with $[w]_{A_1 ^+}\leq \delta$ then for every bounded interval bounded $I=(a,b)$ we have that
	\[
	 \| w \|_{L^{r_w,\infty}(I,\frac{dx}{|I|})}  ^{r_w}  \leq  \M^-(w\ind_I)(b)^{r_w-1} w(I).
	\]
Furthermore, if $w\in A_\infty ^+$ with $[w]_{A_\infty ^+}\leq \delta$ then for every bounded interval $I=(a,b)$ we have that
	\[
	 \| \M^-(w\ind_I) \|_{L^{r_w,\infty}(I,\frac{dx}{|I|})}  ^{r_w}  \leq  \M^- _{[2]}(w\ind_I)(b)^{r_w-1} \M^-(w\ind_I)(I).
	\]
The range of weak-type integrability and the multiplicative constant $1$ are best possible.
\end{theorem}

\begin{proof} We only give the proof of the $A_1 ^+$-case. The proof of the $A_\infty ^+$-case follows similarly, exploiting the estimates established in the proof of Theorem~\ref{t.onesided}.
	
So let us assume that $w\in A_1 ^+$ with $[w]_{A_1 ^+}\leq \delta$. We set $\lambda_o\coloneqq \M^-(w\ind_I)(b)/\delta$ and we remember that from the proof of Theorem~\ref{t.onesidedA1} we have the estimate $w(E_\lambda)\leq \lambda\delta |E_\lambda|$ for $\lambda\geq \lambda_o$, where $E_\lambda\coloneqq \{x\in I:\, w(x)>\lambda\}.$ It is not hard to see that this implies that
	\[
	\fint_0 ^t w^* \leq \delta w^*(t)
	\]
for all $t\in(0,|E_{\lambda_o}|]$, where $w^*$ is the left-continuous non-increasing rearrangement of $w\ind_I$. A variation of the proof of \cite{Wik}*{Lemma 2} then shows that for all $\lambda\geq \lambda_o$ we have
\[
w(E_\lambda)=\int_0 ^{|E_\lambda|} w^* \leq \Big(\frac{|E_\lambda|}{|E_{\lambda_o}|}\Big)^\frac{1}{\delta} \int_0 ^{|E_{\lambda_o}|}w^*= \Big(\frac{|E_\lambda|}{|E_{\lambda_o}|}\Big)^\frac{1}{\delta} w(E_{\lambda_o}).
\]
Using the trivial estimate $w(E_\lambda)>\lambda|E_\lambda|$ the  previous estimate yields
\[
\lambda|E_{\lambda}|^{\frac{\delta-1}{\delta}} \leq \Big(\frac{w(E_{\lambda_o})}{|E_{\lambda_o}|}\Big)^\frac{1}{\delta}w(I)^{1-\frac{1}{\delta}} \leq (\lambda_o \delta)^\frac{1}{\delta}w(I)^\frac{\delta-1}{\delta}= \M^-(w\ind_I)(b)^\frac{1}{\delta}w(I)^\frac{\delta-1}{\delta}.
\]
This is the desired estimate and the sharpness follows as in Theorem~\ref{t.A1endpoint}.
\end{proof}

\subsection{Extensions to general measures}\label{s.genmeasuoned} The one-dimensional proofs of the reverse H\"older inequalities relied mostly on topological properties of the real line rather than properties of the Lebesgue measure. Indeed, one can consider a non-negative Radon measure $\mu$ on $\R$ and derive essentially identical results, as long as $\mu$ doesn't have atoms. To formulate one example let $\mu$ be such a measure and define $A_\infty(\mu)$ to be the set of weights $w$ such that
\[
(w)_{A_\infty(\mu)}\coloneqq \sup_{I:\, \mu(I)>0} \frac{1}{\int_I wd\mu}\int_I \M^\mu(w\ind_I)d\mu <+\infty,
\]
where the maximal function $\M^\mu$ is defined naturally as
\[
\M^\mu f(x)\coloneqq \sup_{\substack{x\in I\\\mu(I)>0}} \frac{1}{\mu(I)}\int_I |f(y)|d\mu(y).
\]
Here the supremum is taken with respect to all open, bounded subintervals of the real line. It is well known that $\M^\mu$ maps $L^1(\mu)$ to $L^{1,\infty}(\mu)$ and $L^p(\mu)$ to itself, in fact uniformly over all such measures. If furthermore $\mu$ has no atoms then $E_\lambda ^\mu\coloneqq \{x\in I:\, \M^\mu f(x)>\lambda\}$ is an open set and Lemma~\ref{l.risingsun2} can be easily extended to give a description of $E_\lambda ^\mu$ as a union of its open connected components which are maximal, as in (i) of Lemma~\ref{l.risingsun2}. An obvious extension of Lemma~\ref{l.master} to general measures then gives the following corollary.

\begin{corollary}\label{c.nondoubling1d} Suppose that $\mu$ is a non-atomic, non-negative Radon measure on $\R$ and $w\in A_\infty(\mu)$ with $(w)_{A_\infty(\mu)}\leq \delta $ for some $\delta>1$ and set $r_w\coloneqq \delta/(\delta-1)$. Then for all $1\leq r<r_w$ and for all finite intervals $I$ on the real line with $\mu(I)>0$ we have
	\[
	\frac{1}{\mu(I)}\int_I w^r d\mu \leq \frac{1}{\delta^{r-1}}\frac{r'-1}{r'-\delta}\big(\frac{1}{\mu(I)}\int_I w d\mu\big)^r.
	\] 
Furthermore, we have for all finite intervals $I$ that
	\[
	 \| w \|_{L^{r_w,\infty}(I,\frac{d\mu}{\mu(I)})}    \leq  \frac{1}{\mu(I)}\int_I w d\mu 
	\]
\end{corollary}

Once a sharp reverse H\"older inequality is proved in one dimension, it would be natural to extend it to higher dimensions for multiparameter $A_p$ weights defined over the basis of rectangular parallelepipeds instead of cubes. This extension us usually achieved by slicing and induction type arguments that can be quite involved. At the end of the next section, as a byproduct of the main Theorem~\ref{t.inftyn}, we will present a very simple approach to sharp reverse H\"older inequalities for multiparameter flat weights.

\section{Asymptotically sharp reverse H\"older inequalities in higher dimensions} \label{s.higher}
\subsection{The proof of Theorem~\ref{t.inftyn}} In this section we give the proof of Theorem \ref{t.inftyn}. Following the structure of the one-dimensional proofs, we will want to treat the function $\M(w\ind_Q)$ as a local $A_1$ weight. However, since in higher dimensions the most efficient covering algorithm we have in our disposal is the Calder\'on-Zygmund decomposition, it will be convenient to work on the dyadic level relative to a given cube $Q$. Thus, given a cube $Q$ we define the local dyadic maximal operator $\MQ$ as
\[
\MQ f(x)\coloneqq \sup_{\substack{S\in \mathcal D(Q)\\S\ni x}}\fint_S |f(y)|dy,
\]
where $\mathcal D(Q)$ is the dyadic grid contained in $Q$ and generated by repeatedly bisecting the sides of $Q$.


The following lemma contains the main estimate for the asymptotically sharp reverse H\"older inequality of Theorem~\ref{t.inftyn}.
\begin{lemma}\label{l.superlevel} Let $w\in A_\infty$ with constant $(w)_{A_\infty}\leq \delta$ for some $\delta>1$. Then for all cubes $Q$ in $\R^n$ with sides parallel to the coordinate axes we have for all $\lambda\geq \lambda_o \coloneqq \M(w\ind_Q)(Q)/(\delta|Q|)$ that
	\[
	\MQ (w\ind_Q)(\{x\in Q:\, \MQ (w\ind_Q)>\lambda\})\leq c_n(\delta) \lambda|\{x\in Q:\, \M(w\ind_Q)>\lambda\}|,
	\]
where $c_n(\delta)\coloneqq \delta+(2^n-1)(\delta -1)$.
\end{lemma}

\begin{proof} We define $E_\lambda \coloneqq \{x\in Q:\, \MQ (w\ind_Q)(x)>\lambda\}$ and for some fixed $\lambda\geq \lambda_o$ let us consider the maximal cubes $\{Q_j\}_j\subseteq \mathcal D(Q)$ such that $\fint_{Q_j} w>\lambda$. If there is only one maximal cube $Q_1= Q$ then we necessarily have that $E_\lambda= Q$ almost everywhere. In this case we can trivially estimate
	\[
	\MQ  (w\ind_Q)(E_\lambda)=\MQ (w\ind_Q)(Q)=\lambda_o\delta |Q|\leq \lambda \delta |E_\lambda|
	\]
as $\lambda\geq \lambda_o$. In the complementary case we have that all the maximal cubes $\{Q_j\}_j$ are strictly contained in $Q$ and $E_\lambda=\bigcup_j Q_j$. Furthermore, by the maximality of the cubes $Q_j$, the maximal function $\M_Q$ can be localized; if $x\in Q_j$ for some $j$ then
\[
\MQ (w\ind_Q)(x) = \MQ(w\ind_{Q_j})(x).
\]
We conclude that
\[
\MQ (w\ind_Q)(E_{\lambda})=\sum_j \MQ (w\ind_{Q_j})(Q_j).
\]
For a cube $R\in \mathcal D(Q)$ let us denote by $R^{(1)}$ its unique dyadic parent. Observe that for each $j$ we have that $w(Q_j)/|Q_j|>\lambda$ and $w(Q_j ^{(1)})/|Q_j ^{(1)}|\leq \lambda$, as all the cubes $Q_j$ are strictly contained in $Q$ and they are maximal. Thus there is a unique cube $Q_j ^*$ such that $Q_j \subsetneq Q_j ^* \subseteq Q_j ^{(1)}$ with $w(Q_j ^*)=\lambda |Q_j ^*|$. Note that $Q_j ^*$ is the dilation of $Q_j$, with respect to the corner that $Q_j$ shares with its parent cube, by some factor $1<\gamma\leq 2$. Then for each $j$ we can estimate
\[
\begin{split}
\MQ(w\ind_{Q_j})(Q_j)&\leq \M(w\ind_{Q_j})(Q_j)\leq \M(w\ind_{Q_j ^*})(Q_j)
\\
&=\M(w\ind_{Q_j ^*})(Q_j ^*)-\M(w\ind_{Q_j ^*})(Q_j ^*\setminus Q_j)
\\
& \leq \delta \lambda |Q_j ^*|-\lambda |Q_j ^*\setminus Q_j|\leq \delta \lambda |Q_j|+(\delta-1)\lambda |Q_j ^*\setminus Q_j|.
\end{split}
\]
In the first line of the estimate above we used that $\M_S(w\ind_S)\leq \M(w\ind_S)$ for any cube $S$, while in the last line of the estimate we used that $\M(w\ind_{Q_j ^*})\geq w(Q_j ^*)/|Q_j ^*|= \lambda $ on $Q_j ^*\setminus Q_j$. Now we use the trivial estimate $|Q_j ^*\setminus Q_j|\leq |Q_j ^{(1)}\setminus Q_j|= (2^n-1)|Q_j|$ to conclude that
\[
\MQ(w\ind_Q)(E_\lambda)\leq  (\delta +(2^n-1)(\delta-1))\lambda |E_\lambda|\eqqcolon c_n(\delta) \lambda |E_\lambda|
\]
as we wanted to show.
\end{proof}
Lemma~\ref{l.superlevel} and Lemma~\ref{l.master} now immediately imply the following result.
\begin{theorem}\label{t.ainftyndyad} Let $w\in A_\infty$ with $(w)_{A_\infty}\leq \delta$ for some $\delta>1$. Then for every cube $Q$ in $\R^n$ and all $1\leq r < 1+\frac{1}{2^n(\delta-1)}$ we have the reverse H\"older inequality
	\[
	\fint_Q \MQ(w\ind_Q)^r \leq \frac{1}{\delta^{r-1}} \frac{r'-1}{r'-1-2^n(\delta-1)}\big(\fint_Q \MQ(w\ind_Q) \big)^r.
	\]
\end{theorem}
The proof of Theorem~\ref{t.inftyn} is now an easy corollary of Theorem~\ref{t.ainftyndyad}.
\begin{proof}[Proof of Theorem~\ref{t.inftyn}] By the definition of $(w)_{A_\infty}$ we have that $(\fint_Q \MQ(w\ind_Q))^r\leq \delta ^r w_Q ^r$ and obviously we have that $w\ind_Q \leq \M_Q(w\ind_Q)\ind_Q$. With these observations, Theorem~\ref{t.inftyn} follows immediately from Theorem~\ref{t.ainftyndyad} above.
\end{proof}

\subsection{Extensions to the multiparameter setting} In dimensions $n>1$ we used the Calder\'on-Zygmund decomposition and dyadic methods instead of the more precise covering argument of Lemma~\ref{l.risingsun2}  for the proof of the reverse H\"older inequality. This resulted to the appearance of dimensional constants in our estimate for the local integrability exponent in the reverse H\"older inequality of Theorem~\ref{t.inftyn}, that are most probably not optimal.

It turns out that the proof of Theorem~\ref{t.inftyn} can be adjusted to the geometry of \emph{multiparameter} or \emph{strong} weights. Reverse H\"older inequalities for multiparameter weights have been studied for example in \cite{Kin1} and also in \cites{HaPa,LPR} while in \cite{DMRO-Ainfty} the authors study several properties and equivalent definitions of $A_\infty$-weights defined with respect to general bases of convex sets. The results \cite{HaPa} also provide reverse H\"older inequalities for $A_{\infty}^*$ by means of Solyanik estimates, but these are not asymptotically sharp.
 
Here we will work directly with multiparameter weights defined with respect to more general measures. Given a non-negative Radon measure $\mu$ on $\R^n$ we define the class of weights $A^*_\infty(\mu)$ to be collection of all weights $w$ on $\R^n$, such that
\[
(w)_{A^*_\infty(\mu)}\coloneqq \sup_{R:\, \mu(R)>0}\frac{1}{\int_R w d\mu}\int_{R}\M_s(w\ind_R) d\mu<+\infty.
\]
where $\M_s ^\mu$ denotes the \emph{strong maximal operator} with respect to $\mu$
\[
\M_s ^\mu f(x)\coloneqq \sup_{\substack{R\ni x\\\mu(R)>0}} \frac{1}{\mu(R)}\int_R  |f|\ d\mu \qquad x\in\R^n.
\]
The suprema above are taken with respect to all rectangular parallelepipeds $R$ in $\mathbb R^n$ with sides parallel to the coordinate axes. 
 
The main idea here is that the proof of Theorem \ref{t.inftyn} relies on three fundamental properties. One of them is the dyadic structure of the basis defining $\M_Q$ and the doubling property of the Lebesgue measure. The second fundamental property used above is that the \emph{dyadic} maximal function can be localized on the maximal cubes from the Calder\'on-Zygmund decomposition. And we also use a somehow trivial control of the weight by the maximal function, namely that $w\le \M_Qw$ for a.e. $x\in Q$ which in turn is a consequence of the Lebesgue differentiation theorem and therefore it also depends on the geometry of the dyadic grid. 

We have the following corollary for non-atomic product measures, which clearly includes the classical case of multiparameter weights with respect to the Lebesgue measure on $\mathbb{R}^n$. We can go beyond doubling measures as in Coroallry~\ref{c.nondoubling1d} but here we can jump immediately to higher dimensions.

\begin{corollary}\label{c.mu-dyadic-Rn} Let $\mu\coloneqq \bigotimes_{i=1}^n \mu_i$, where $\mu_1,\mu_2,\dots, \mu_n$ are defined on $\mathbb{R}$ and each $\mu_i$ is a non-negative, non-atomic Radon measure on the real line.  Let $w\in A^*_\infty(\mu)$ with $(w)_{A_\infty ^* (\mu)}\leq \delta$ for some $\delta\geq 1$. Then for every rectangular parallelepiped $R$ in $\R^n$ with sides parallel to the coordinate axes such that $\mu(R)>0$ and all $1\leq r <1+\frac{1}{2^n(\delta-1)}$ we have the reverse H\"older inequality
\[
\frac{1}{\int_R w d\mu}\int_R w^r d\mu \leq \delta\frac{r'-1}{r'-1-2^n(\delta-1)} \big(\frac{1}{\int_R wd\mu}\int_Q w d\mu\big)^r.
\]
 \end{corollary}

\begin{proof}
We start with the one dimensional case by constructing a $\mu$-dyadic grid. Given an interval $I$ with $\mu(I)>0$, define the first generation $G_1(I)$ of the dyadic grid as the collection $\{I_-,I_+\}$ where $I_+$, $I_-$ are subintervals of $I$ with disjoint interiors that satisfy $\mu(I_+) = \mu(I_-) = \mu(I)/2$. Note that this splitting $I=I_- \cup I_+$ is in general non-unique. For specificity we always choose $I_-$ to have minimal Lebesgue measure among the allowed intervals $I_-$. Define $G_2(I)=G_1(I_+)\cup G_1(I_-)$ and recursively define the next generations consisting of closed intervals with disjoint interiors. Let $\mathcal D^\mu(I)$ be the family of all the dyadic intervals generated with this procedure. A collection of nested intervals from this grid will be called a \emph{chain}. More precisely, a chain $\mathcal{C}$ will be of the form $\mathcal{C}=\{J_i\}_{i\in\mathbb N}$ such that $J_i\in G_i(I)$, and $J_{i+1}\subset J_i$ for all $i\ge1$.

If we define $\mathcal{C}_\infty:=\bigcap_{J\in \mathcal{C}} J$ as the \emph{limit set} of the chain $\mathcal{C}$, we have that $\mathcal{C}_\infty$ could be a single point or a closed interval of positive length. In any case, we clearly have that $\mu(\mathcal{C}_\infty)=0$. We will say that  those limit sets $\mathcal{C}_\infty$ of positive length are \emph{removable}. These are at most countably many  and thus their union is also a $\mu$-null set. We denote by $\mathcal{R}$ the set of all chains with removable limits. If we define
\begin{equation}\label{eq.removed-line}
E:= I\setminus \bigcup_{\mathcal{C}\in \mathcal{R}}\mathcal{C_\infty}
\end{equation}
we conclude that $\mu(I)=\mu(E)$ and, in addition, for any $x\in E$ there exists a chain of nested intervals shrinking to $x$. Therefore the grid $\mathcal D^\mu (I)$ forms a differentiation basis on $E$. Moreover, the dyadic structure of the basis guarantees the Vitali covering property, see \cite{Guzman-diff}*{Ch.1}, and therefore this basis differentiates $L^1(E)$.

We define a \emph{dyadic} maximal operator associated with this grid as follows. For any $x\in E$ let
\[
\M^\mu _I f(x) \coloneqq \sup_{\substack{ J\in \mathcal D ^\mu(I) \\J\ni x }}\frac{1}{\mu(J)}\int_J |f|\ d\mu.
\]
 By a standard differentiation argument, we have that this maximal function satisfies $|f|\le \M^\mu_I f$ almost everywhere on $E$. The dyadic structure also guarantees that $\M^\mu_I$ can be localized on maximal Calder\'on-Zygmund intervals. As a corollary of the method of proof of Theorem~\ref{t.inftyn} and the discussion above, the one dimensional case of Corollary \ref{c.mu-dyadic-Rn} is proved.

For the higher dimensional case we remember that $\mu=\bigotimes_{i=1}^n \mu_i$, where $\mu_1,\mu_2,\dots, \mu_n$ are defined on $\mathbb{R}$ and none of them has atoms.  We build an appropriate $\mu$-dyadic grid relative to a fixed rectangular parallelepiped $R$ in $\R^n$. Suppose that the rectangular parallelepiped $R$ is of the form $R=\prod_{i=1}^n I_i$. We perform the partition on each direction in order to obtain the dyadic grid $\mathcal D^{\mu_i}=\bigcup_{j\ge 1}G_j(I_i)$. Following the same idea as in the one-dimensional case we call $\mathcal{R}_i$ the family of all chains with \emph{removable} limits in each direction. After removing all of them, we can assume that any chain $\mathcal{C}=\{J_m\}_{m\in \mathbb{N}}$ in $\mathcal{D}^{\mu_i}$ verifies that $\lim_{m\to \infty}\text{diam}(J_m)=0$. As in \eqref{eq.removed-line} we define the sets
\[
E_i:= I_i\setminus \bigcup_{\mathcal{C}\in \mathcal{R}_i}\mathcal{C_\infty},\qquad 1\le i\le n,
\]
and $E:=E_1\times\cdots\times E_n.$ We can build the dyadic grid for $R$ by taking  products elements of each $\mathcal{D}^{\mu_i}$ of the same generation. More precisely, the $k$-th generation dyadic grid is
\[
\mathcal{D}_k^\mu(R)\coloneqq \left\{J_1\times \cdots \times J_n:\, J_i\in G_k(I_i), 1\le i\le n \right\},
\]
and the full grid is the union of all generations
\[
\mathcal{D}^\mu(R)=\bigcup_k \mathcal{D}_k^\mu(R).
\]
The grid $\mathcal{D}^\mu (R)$ defined in this way is a differentiation basis on $E$, as in the 1-dimensional case, satisfying the Vitali covering property. Hence,  the same reasoningallows us to conclude that the maximal operator $\M_R^\mu$, defined as
\[
\M^\mu _{R}f(x)\coloneqq \sup_{\substack{S\in \mathcal D^\mu(R)\\S\ni x}} \frac{1}{\mu(S)}\int_S |f|d\mu,
\]
satisfies the inequality  $w(x)\le \M^\mu _R  w(x) $ for $\mu$-a.e. $x\in E$. We also have the localization property for $\M^\mu_R$ and therefore we can proceed as in the proof of Theorem \ref{t.inftyn} in order to conclude the reverse H\"older inequality for multiparameter weights with respect non atomic product measures.
 \end{proof}

\section{Acknowledgment}
This work started while E.R. was visiting the Basque Center for Applied Mathematics in Bilbao, Spain. The second author is deeply indebted to all the staff at the institute, and in particular to Carlos P\'erez, for their hospitality and generosity.


\end{document}